\documentclass[10pt]{article}
\usepackage{amsmath,amsthm,amsfonts,latexsym}
\input{T2Aenc.def}
\usepackage[utf8]{inputenc}
\newenvironment{cyr}
  {\begingroup\fontencoding{T2A}\fontfamily{cmr}\fontseries{m}\fontshape{n}\selectfont}
  {\endgroup}

\allowdisplaybreaks[4]

\newcommand{\Extra}[1]{}

\usepackage[pdfpagemode=UseNone,pdfstartview=FitH]{hyperref}

\newtheorem{theorem}{Theorem}
\newtheorem{proposition}[theorem]{Proposition}

\newtheorem{corollary}[theorem]{Corollary}

\theoremstyle{definition}

\theoremstyle{remark}
\newtheorem{remark}[theorem]{Remark}

\newcommand{\dd}{\,\mathrm{d}}
\renewcommand{\d}{\mathrm{d}}

\DeclareMathOperator{\infproj}{proj^{\inf}_{\Omega}}
\DeclareMathOperator{\supproj}{proj^{\sup}_{\Omega}}

\DeclareMathOperator{\conf}{conf}

\newcommand{\FFF}{\mathcal{F}}

\newcommand{\PPP}{\mathcal{P}}
\newcommand{\EEE}{\mathcal{E}}

\newcommand{\bin}{\mathrm{bin}}

\newcommand{\EBern}{\mathcal{E}_\mathrm{Bern}}

\newcommand{\Eiid}{\mathcal{E}_\mathrm{iid}}

\newcommand{\Eexch}{\mathcal{E}_\mathrm{exch}}

\newcommand{\Ebin}{\mathcal{E}_\mathrm{bin}}

\newcommand{\Econf}{\mathcal{E}_\mathrm{conf}}

\newcommand{\Esin}{\mathcal{E}_\mathrm{sin}}

\newcommand{\PBern}{\mathcal{P}_\mathrm{Bern}}
\newcommand{\Pexch}{\mathcal{P}_\mathrm{exch}}
\newcommand{\Pbin}{\mathcal{P}_\mathrm{bin}}
\newcommand{\Psin}{\mathcal{P}_\mathrm{sin}}

\title{Non-algorithmic theory of randomness}
\author{Vladimir Vovk}

\begin{document}
\maketitle

\begin{abstract}
  This paper proposes an alternative language
  for expressing results of the algorithmic theory of randomness.
  The language is more precise in that it does not involve unspecified additive or multiplicative constants,
  making mathematical results, in principle, applicable in practice.
  Our main testing ground for the proposed language is the problem of defining Bernoulli sequences,
  which was of great interest to Andrei Kolmogorov and his students.

   The version of this paper at \url{http://alrw.net} (Working Paper 25)
   is updated most often.
\end{abstract}

\section{Introduction}

There has been a great deal of criticism of the notion of p-value lately,
and in particular,
Glenn Shafer \cite{Shafer:arXiv1903} defended the use of betting scores instead.
This paper refers to betting scores as e-values
and demonstrates their advantages by establishing results that become much more precise
when they are stated in terms of e-values instead of p-values.

Both p-values and e-values have been used, albeit somewhat implicitly,
in the algorithmic theory of randomness:
Martin-L\"of's tests of algorithmic randomness \cite{Martin-Lof:1966} are an algorithmic version of p-functions
(i.e., functions producing p-values \cite{Gurevich/Vovk:2019COPA})
while Levin's tests of algorithmic randomness \cite{Levin:1976-local,Gacs:2005} are an algorithmic version of e-functions
(this is the term we will use in this paper for functions producing e-values).
Levin's tests are a natural modification of Martin-L\"of's tests
leading to simpler mathematical results;
similarly, many mathematical results stated in terms of p-values
become simpler when stated in terms of e-values.

The algorithmic theory of randomness is a powerful source of intuition,
but strictly speaking, its results are not applicable in practice
since they always involve unspecified additive or multiplicative constants.
The goal of this paper is to explore ways of obtaining results that are more precise;
in particular, results that may be applicable in practice.
The price to pay is that our results may involve more quantifiers
(usually hidden in our notation)
and, therefore, their statements may at first appear less intuitive.

In Section~\ref{sec:simple} we define p-functions and e-functions
in the context of testing simple statistical hypotheses,
explore relations between them,
and explain the intuition behind them.
In Section~\ref{sec:composite} we generalize these definitions, results, and explanations
to testing composite statistical hypotheses.

Section~\ref{sec:Bayes} is devoted to testing in Bayesian statistics
and gives non-algorithmic results that are particularly clean and intuitive.
They will be used as technical tools later in the paper.
In Section~\ref{sec:para-Bayesian} these results are slightly extended
and then applied to clarifying the difference
between statistical randomness and exchangeability.
(In this paper we use ``statistical randomness''
to refer to being produced by an IID probability measure;
there will always be either ``algorithmic'' or ``statistical''
standing next to ``randomness''
in order to distinguish between the two meanings.)

Section~\ref{sec:Kolmogorov} explores the question of defining Bernoulli sequences,
which was of great interest to Kolmogorov \cite{Kolmogorov:1968},
Martin-L\"of \cite{Martin-Lof:1966}, and Kolmogorov's other students.
Kolmogorov defined Bernoulli sequences as exchangeable sequences,
but we will see that another natural definition is narrower than exchangeability.

Kolmogorov paid particular attention to algorithmic randomness
w.r.\ to uniform probability measures on finite sets.
On one hand, he believed that his notion of algorithmic randomness in this context ``can be regarded as definitive''
\cite{Kolmogorov:1983LNM-local},
and on the other hand, he never seriously suggested any generalizations of this notion
(and never endorsed generalizations proposed by his students).
In Section~\ref{sec:Kolmogorov} we state a simple result in this direction
that characterizes the difference between Bernoulliness and exchangeability.

In Sections~\ref{sec:Bayes} and~\ref{sec:Kolmogorov}
we state our results first in terms of e-functions and then p-functions.
Results in terms of e-functions are always simpler and cleaner,
supporting Glenn Shafer's recommendation in \cite{Shafer:arXiv1903}
to use betting scores more widely.

\begin{remark}
  There is no standard terminology for what we call e-values and e-functions.
  In addition to Shafer's use of ``betting scores'' for our e-values,
  \begin{itemize}
  \item
    Gr\"unwald et al.\ \cite{Grunwald/etal:arXiv1906} refer to e-values as ``s-values''
    (``s'' for ``safe''),
  \item
    and Gammerman and Vovk \cite{Gammerman/Vovk:2003} refer to the inverses of e-values as ``i-values''
    (``i'' for ``integral'').
  \end{itemize}
\end{remark}

No formal knowledge of the algorithmic theory of randomness will be assumed
in this paper;
the reader can safely ignore all comparisons
between our results and results of the algorithmic theory of randomness.

\subsection*{Notation}

Our notation will be mostly standard or defined at the point where it is first used.
If $\FFF$ is a class of $[0,\infty]$-valued functions on some set $\Omega$
and $g:[0,\infty]\to[0,\infty]$ is a function,
we let $g(\FFF)$ stand for the set of all compositions $g(f)=g\circ f$, $f\in\FFF$
(i.e, $g$ is applied to $\FFF$ element-wise).
We will also use obvious modifications of this definition:
e.g., $0.5\FFF^{-0.5}$ would be interpreted as $g(\FFF)$,
where $g(u):=0.5 u^{-0.5}$ for $u\in[0,\infty]$.

\section{Testing simple statistical hypotheses}
\label{sec:simple}

Let $P$ be a probability measure on a measurable space $\Omega$.
A \emph{p-function} \cite{Gurevich/Vovk:2019COPA} is a measurable function $f:\Omega\to[0,1]$ such that,
for any $\epsilon>0$, $P\{f\le\epsilon\} \le \epsilon$.
An \emph{e-function} is a measurable function $f:\Omega\to[0,\infty]$ such that
$\int f \dd P\le1$.
(E-functions have been promoted in \cite{Shafer:arXiv1903}, \cite{Grunwald/etal:arXiv1906},
and \cite[Section~11.5]{Shafer/Vovk:2019},
but using different terminology.)

Let $\PPP_P$ be the class of all p-functions and $\EEE_P$ be the class of all e-functions,
where the underlying measure $P$ is shown as subscript.
We can define p-values and e-values as values taken by p-functions and e-functions,
respectively.
The intuition behind p-values and e-values will be discussed later in this section.

The following is an algorithm-free version of the standard relation
(see, e.g., \cite[Lemma 4.3.5]{Li/Vitanyi:2008})
between Martin-L\"of's and Levin's algorithmic notions of randomness deficiency.
\begin{proposition}\label{prop:p-vs-e}
  For any probability measure $P$ and $\kappa\in(0,1)$,
  \begin{equation}\label{eq:p-vs-e}
    \kappa \PPP_P^{\kappa-1} \subseteq \EEE_P \subseteq \PPP_P^{-1}.
  \end{equation}
\end{proposition}

\begin{proof}
  The right inclusion in~\eqref{eq:p-vs-e} follows from the Markov inequality:
  if $f$ is an e-function,
  \[
    P\{f^{-1}\le\epsilon\}
    =
    P\{f\ge1/\epsilon\}
    \le
    \epsilon.
  \]

  The left inclusion in~\eqref{eq:p-vs-e} follows from \cite[Section~11.5]{Shafer/Vovk:2019}.
  The value of the constant in front of the $\PPP_P^{\kappa-1}$ on the left-hand side of~\eqref{eq:p-vs-e}
  follows from $\int_0^1 p^{\kappa-1}\dd p=1/\kappa$.
\end{proof}

Both p-functions and e-functions can be used for testing statistical hypotheses.
In this section we only discuss \emph{simple statistical hypotheses},
i.e., probability measures.
Observing a large e-value or a small p-value w.r.\ to a simple statistical hypothesis $P$
entitles us to rejecting $P$ as the source of the observed data,
provided the e-function or p-function were chosen in advance.
The e-value can be interpreted as the amount of evidence against $P$
found by our chosen e-function.
Similarly, the p-value reflects the amount of evidence against $P$ on a different scale;
small p-values reflect a large amount of evidence against $P$.

\begin{remark}\label{rem:assumptions}
  Proposition~\ref{prop:p-vs-e} tells us that using p-values and using e-values
  are equivalent, on a rather crude scale.
  Roughly, a p-value of $p$ corresponds to an e-value of $1/p$.
  The right inclusion in~\eqref{prop:p-vs-e} says that any way of producing e-values $e$
  can be translated into a way of producing p-values $1/e$.
  On the other hand,
  the left inclusion in~\eqref{prop:p-vs-e} says that any way of producing p-values $p$
  can be translated into a way of producing e-values $\kappa p^{\kappa-1}\approx1/p$,
  where the ``$\approx$'' assumes that we are interested in the asymptotics as $p\to0$,
  $\kappa>0$ is small,
  and we ignore positive constant factors (as customary in the algorithmic theory of randomness).
\end{remark}

\begin{remark}\label{rem:p-vs-e}
  Proposition~\ref{prop:p-vs-e} can be greatly strengthened,
  under the assumptions of Remark~\ref{rem:assumptions}.
  For example, we can replace \eqref{eq:p-vs-e} by
  \begin{equation*}
    H_\kappa(\PPP_P) \subseteq \EEE_P \subseteq \PPP_P^{-1},
  \end{equation*}
  where
  \begin{equation}\label{eq:H}
    H_\kappa(v)
    :=
    \begin{cases}
      \infty & \text{if $v=0$}\\
      \kappa (1+\kappa)^\kappa v^{-1} (-\ln v)^{-1-\kappa} & \text{if $v\in(0,e^{-1-\kappa}]$}\\
      0 & \text{if $v\in(e^{-1-\kappa},1]$}
    \end{cases}
  \end{equation}
  and $\kappa\in(0,\infty)$
  (see \cite[Section~11.1]{Shafer/Vovk:2019}).
  The value of the coefficient $\kappa (1+\kappa)^\kappa$ in \eqref{eq:H} follows from
  \[
    \int_0^{e^{-1-\kappa}}
    v^{-1} (-\ln v)^{-1-\kappa}
    \dd v
    =
    \frac{1}{\kappa(1+\kappa)^\kappa}.
  \]
\end{remark}

\section{Testing composite statistical hypotheses}
\label{sec:composite}

Let $\Omega$ be a measurable space,
which we will refer to as our \emph{sample space},
and $\Theta$ be another measurable space (our \emph{parameter space}).
We say that $P=(P_{\theta}\mid\theta\in\Theta)$ is a \emph{statistical model}
on $\Omega$
if $P$ is a Markov kernel with source $\Theta$ and target $\Omega$:
each $P_\theta$ is a probability measure on $\Omega$,
and for each measurable $A\subseteq\Omega$,
the function $P_{\theta}(A)$ of $\theta\in\Theta$ is measurable.

The notions of an e-function and a p-function each split in two.
We are usually really interested only in the outcome $\omega$,
while the parameter $\theta$ is an auxiliary modelling tool.
This motivates the following pair of simpler definitions.
A measurable function $f:\Omega\to[0,\infty]$
is an \emph{e-function} w.r.\ to the statistical model $P$
(which is our \emph{composite statistical hypothesis} in this context)
if
\[
  \forall\theta\in\Theta:
  \int_{\Omega} f(\omega) P_{\theta}(\d\omega) \le 1.
\]
In other words, if $P^*(f)\le1$,
where $P^*$ is the upper envelope
\begin{equation}\label{eq:upper}
  P^*(f)
  :=
  \sup_{\theta\in\Theta}
  \int f(\omega) P_{\theta}(\d\omega)
\end{equation}
(in Bourbaki's \cite[IX.1.1]{Bourbaki:integration} terminology,
$P^*$ is an encumbrance provided the integral in \eqref{eq:upper}
is understood as the upper integral).
Similarly, a measurable function $f:\Omega\to[0,1]$
is a \emph{p-function} w.r.\ to the statistical model $P$ if,
for any $\epsilon>0$,
\[
  \forall\theta\in\Theta:
  P_{\theta}\{\omega\in\Omega\mid f(\omega)\le\epsilon\} \le \epsilon.
\]
In other words, if, for any $\epsilon>0$, $P^*(1_{\{f\le\epsilon\}})\le\epsilon$.

Let $\EEE_P$ be the class of all e-functions w.r.\ to the statistical model $P$,
and $\PPP_P$ be the class of all p-functions w.r.\ to $P$.
We can easily generalize Proposition~\ref{prop:p-vs-e}
(the proof stays the same).

\begin{proposition}\label{prop:p-vs-e-stat}
  For any statistical model $P$ and $\kappa\in(0,1)$,
  \begin{equation*}
    \kappa \PPP_P^{\kappa-1} \subseteq \EEE_P \subseteq \PPP_P^{-1}.
  \end{equation*}
\end{proposition}

For $f\in\EEE_P$, we regard the e-value $f(\omega)$
as the amount of evidence against the statistical model $P$
found by $f$ (which must be chosen in advance) when the outcome is $\omega$.
The interpretation of p-values is similar.

In some case we would like to take the parameter $\theta$ into account
more seriously.
A measurable function $f:\Omega\times\Theta\to[0,\infty]$
is a \emph{conditional e-function} w.r.\ to the statistical model $P$ if
\[
  \forall\theta\in\Theta:
  \int_{\Omega} f(\omega;\theta) P_{\theta}(\d\omega) \le 1.
\]
Let $\bar\EEE_P$ be the class of all such functions.
And a measurable function $f:\Omega\times\Theta\to[0,1]$
is a \emph{conditional p-function} w.r.\ to $P$ if
\[
  \forall\epsilon>0\:
  \forall\theta\in\Theta:
   P_{\theta}
   \left\{
     \omega\in\Omega
     \mid
     f(\omega;\theta) \le \epsilon
   \right\}
   \le
   \epsilon.
\]
Let $\bar\PPP_P$ be the class of all such functions.

We can embed $\EEE_P$ (resp.\ $\PPP_P$)
into $\bar\EEE_P$ (resp.\ $\bar\PPP_P$)
by identifying a function $f$ on domain $\Omega$
with the function $f'$ on domain $\Omega\times\Theta$
that does not depend on $\theta\in\Theta$,
$f'(\omega;\theta):=f(\omega)$.

For $f\in\bar\EEE_P$, we can regard $f(\omega;\theta)$
as the amount of evidence against the specific probability measure $P_\theta$
in the statistical model $P$ found by $f$ when the outcome is $\omega$.

We can generalize Proposition~\ref{prop:p-vs-e-stat} further as follows.

\begin{proposition}\label{prop:p-vs-e-bar}
  For any statistical model $P$ and $\kappa\in(0,1)$,
  \begin{equation}\label{eq:p-vs-e-bar}
    \kappa \bar\PPP_P^{\kappa-1} \subseteq \bar\EEE_P \subseteq \bar\PPP_P^{-1}.
  \end{equation}
\end{proposition}

Remarks~\ref{rem:assumptions} and~\ref{rem:p-vs-e} are also applicable
in the context of Propositions~\ref{prop:p-vs-e-stat} and~\ref{prop:p-vs-e-bar}.

\section{The validity of Bayesian statistics}
\label{sec:Bayes}

In this section we establish the validity of Bayesian statistics in our framework,
mainly as a sanity check.
We will translate the results in \cite{Vovk/Vyugin:1993},
which are stated in terms of the algorithmic theory of randomness,
to our algorithm-free setting.
It is interesting that the proofs simplify radically,
and become almost obvious.
(And remarkably, one statement also simplifies.)

Let $P=(P_{\theta}\mid\theta\in\Theta)$ be a statistical model,
as in the previous section,
and $Q$ be a probability measure on the parameter space $\Theta$.
Together, $P$ and $Q$ form a \emph{Bayesian model},
and $Q$ is known as the \emph{prior measure} in this context.

The joint probability measure $T$ on the measurable space $\Omega\times\Theta$ is defined by
\[
  T(A\times B)
  :=
  \int_B P_{\theta}(A) Q(\d\theta),
\]
for all measurable $A\subseteq\Omega$ and $B\subseteq\Theta$.
Let $Y$ be the marginal distribution of $T$ on $\Omega$:
for any measurable $A\subseteq\Omega$, $Y(A):=T(A\times\Theta)$.

The \emph{product} $\bar\EEE_P\EEE_Q$ of $\bar\EEE_P$ and $\EEE_Q$ is defined
as the class of all measurable functions $f:\Omega\times\Theta\to[0,\infty]$ such that,
for some $g\in\bar\EEE_P$ and $h\in\EEE_Q$,
\begin{equation}\label{eq:product}
  f(\omega,\theta)
  =
  g(\omega;\theta)
  h(\theta)
  \quad
  \text{$T$-a.s.}
\end{equation}
Such $f$ can be regarded as ways of finding evidence against $(\omega,\theta)$
being produced by the Bayesian model $(P,Q)$:
to have evidence against $(\omega,\theta)$ being produced by $(P,Q)$
we need to have evidence against $\theta$ being produced by the prior measure $Q$
or evidence against $\omega$ being produced by $P_{\theta}$;
we combine the last two amounts of evidence by multiplying them.
The following proposition tells us that this product
is precisely the amount of evidence against $T$
found by a suitable e-function.

\begin{proposition}\label{prop:VV-1}
  If $(P_\theta\mid\theta\in\Theta)$ is a statistical model with a prior probability measure $Q$ on $\Theta$,
  and $T$ is the joint probability measure on $\Omega\times\Theta$,
  then
  \begin{equation}\label{eq:VV-1}
    \EEE_T
    =
    \bar\EEE_P\EEE_Q.
  \end{equation}
\end{proposition}

Proposition~\ref{prop:VV-1} will be deduced from Theorem~\ref{thm:VV}
in Section~\ref{sec:para-Bayesian}.
It is the analogue of Theorem~1 in \cite{Vovk/Vyugin:1993},
which says, in the terminology of that paper,
that the level of impossibility of a pair $(\theta,\omega)$ w.r.\ to the joint probability measure $T$
is the product of the level of impossibility of $\theta$ w.r.\ to the prior measure $Q$
and the level of impossibility of $\omega$ w.r.\ to the probability measure $P_{\theta}$.
In an important respect, however,
Proposition~\ref{prop:VV-1} is simpler than Theorem~1 in \cite{Vovk/Vyugin:1993}:
in the latter, the level of impossibility of $\omega$ w.r.\ to $P_{\theta}$
has to be conditional on the level of impossibility of $\theta$ w.r.\ to $Q$,
whereas in the former there is no such conditioning.
Besides, Proposition~\ref{prop:VV-1} is more precise:
it does not involve any constant factors (specified or otherwise).

\begin{remark}
  The non-algorithmic formula \eqref{eq:VV-1} being simpler
  than its counterpart in the algorithmic theory of randomness
  is analogous to the non-algorithmic formula $H(x,y)=H(x)+H(y\mid x)$
  being simpler than its counterpart $K(x,y)=K(x)+K(y\mid x,K(x))$
  in the algorithmic theory of complexity,
  $H$ being entropy and $K$ being prefix complexity.
  The fact that $K(x,y)$ does not coincide with $K(x)+K(y\mid x)$
  to within an additive constant, $K$ being Kolmogorov complexity,
  was surprising to Kolmogorov and wasn't noticed for several years
  \cite{Kolmogorov:1968-talk,Kolmogorov:1968}.
\end{remark}

The \emph{inf-projection} onto $\Omega$ of an e-function $f\in\EEE_T$ w.r.\ to $T$
is the function $(\infproj f):\Omega\to[0,\infty]$ defined by
\[
  \left(
    \infproj f
  \right)
  (\omega)
  :=
  \inf_{\theta\in\Theta}
  f(\omega,\theta).
\]
Intuitively, $\infproj f$ regards $\omega$ as typical under the model
if it can be extended to a typical $(\omega,\theta)$ for at least one $\theta$.
Let $\infproj\EEE_T$ be the set of all such inf-projections.

The results in the rest of this section become simpler
if the definitions of classes $\EEE$ and $\PPP$ are modified slightly:
we drop the condition of measurability on their elements
and replace all integrals by upper integrals
and all measures by outer measures.
We will use the modified definitions only in the rest of this section
(we could have used them in the whole of this paper,
but they become particularly useful here since projections of measurable functions
do not have to be measurable \cite{Suslin:1917}).

\begin{proposition}\label{prop:VV-2}
  If $T$ is a probability measure on $\Omega\times\Theta$
  and $Y$ is its marginal distribution on $\Omega$,
  \begin{equation}\label{eq:VV-2}
    \EEE_Y
    =
    \infproj\EEE_T.
  \end{equation}
\end{proposition}

\begin{proof}
  To check the inclusion ``$\subseteq$'' in~\eqref{eq:VV-2},
  let $g\in\EEE_Y$, i.e., $\int g(\omega) Y(\d\omega) \le 1$.
  Setting $f(\omega,\theta):=g(\omega)$,
  we have $\int f(\omega,\theta) T(\d\omega,\d\theta) \le 1$
  (i.e., $f\in\EEE_T$)
  and $g$ is the inf-projection of $f$ onto $\Omega$.

  To check the inclusion ``$\supseteq$'' in~\eqref{eq:VV-2},
  let $f\in\EEE_T$ and $g:=\infproj f$.
  We then have
  \begin{equation*}
    \int g(\omega) Y(\d\omega)
    =
    \int g(\omega) T(\d\omega,\d\theta)
    \le
    \int f(\omega,\theta) T(\d\omega,\d\theta)
    \le
    1.
    \qedhere
  \end{equation*}
\end{proof}

Proposition~\ref{prop:VV-2} says that we can acquire evidence against an outcome $\omega$
being produced by the Bayesian model $(P,Q)$
if and only if we can acquire evidence against $(\omega,\theta)$
being produced by the model for all $\theta\in\Theta$.

We can combine Propositions~\ref{prop:VV-1} and~\ref{prop:VV-2} obtaining
\begin{equation*}
  \EEE_Y
  =
  \infproj
  \left(
    \bar\EEE_P\EEE_Q
  \right).
\end{equation*}
The rough interpretation is that we can acquire evidence against $\omega$ being produced by $Y$
if and only if we can, for each $\theta\in\Theta$,
acquire evidence against $\theta$ being produced by $Q$
or acquire evidence against $\omega$ being produced by $P_\theta$.

The following statements in terms of p-values are cruder,
but their interpretation is similar.

\begin{corollary}\label{cor:VV-1}
  If $\kappa\in(0,1)$ and $(P,Q)$ is a Bayesian model,
  \[
    \kappa^{-1}
    \PPP_T^{1-\kappa}
    \subseteq
    \bar\PPP_P \PPP_Q
    \subseteq
    \kappa^{\frac{2}{1-\kappa}}
    \PPP_T^{\frac{1}{1-\kappa}}.
  \]
\end{corollary}

\begin{proof}
  We can rewrite \eqref{eq:p-vs-e-bar} in Proposition~\ref{prop:p-vs-e-bar} as
  \begin{equation*}
    \kappa^{-1} \bar\PPP_P^{1-\kappa} \subseteq \bar\EEE_P^{-1} \subseteq \bar\PPP_P
  \end{equation*}
  and as
  \begin{equation*}
    \bar\EEE_P^{-1} \subseteq \bar\PPP_P \subseteq \kappa^{\frac{1}{1-\kappa}} \bar\EEE_P^{-\frac{1}{1-\kappa}},
  \end{equation*}
  with similar representations for \eqref{prop:p-vs-e} in Proposition~\ref{prop:p-vs-e}
  and \eqref{prop:p-vs-e-stat} in Proposition~\ref{prop:p-vs-e-stat}.
  Therefore, by \eqref{eq:VV-1} in Proposition~\ref{prop:VV-1},
  \[
    \kappa^{-1}
    \PPP_T^{1-\kappa}
    \subseteq
    \EEE_T^{-1}
    =
    (\bar\EEE_P\EEE_Q)^{-1}
    =
    \bar\EEE_P^{-1} \EEE_Q^{-1}
    \subseteq
    \bar\PPP_P \PPP_Q
  \]
  and
  \[
    \bar\PPP_P \PPP_Q
    \subseteq
    \kappa^{\frac{2}{1-\kappa}}
    \left(
      \bar\EEE_P \EEE_Q
    \right)^{-\frac{1}{1-\kappa}}
    =
    \kappa^{\frac{2}{1-\kappa}}
    \EEE_T^{-\frac{1}{1-\kappa}}
    \subseteq
    \kappa^{\frac{2}{1-\kappa}}
    \PPP_T^{\frac{1}{1-\kappa}}.
    \qedhere
  \]
\end{proof}

\begin{corollary}\label{cor:VV-2}
  If $\kappa\in(0,1)$, $T$ is a probability measure on $\Omega\times\Theta$,
  and $Y$ is its marginal distribution on $\Omega$,
  \[
    \kappa^{-1}
    \supproj\PPP_T^{1-\kappa}
    \subseteq
    \PPP_Y
    \subseteq
    \kappa^{\frac{1}{1-\kappa}}
    \supproj\PPP_T^{\frac{1}{1-\kappa}},
  \]
  where $\supproj$ is defined similarly to $\infproj$
  (with $\sup$ in place of $\inf$).
\end{corollary}

\begin{proof}
  As in the proof of Corollary~\ref{cor:VV-1},
  we have
  \[
    \kappa^{-1}
    \supproj\PPP_T^{1-\kappa}
    \subseteq
    \supproj\EEE_T^{-1}
    =
    \EEE_Y^{-1}
    \subseteq
    \PPP_Y
  \]
  and
  \[
    \PPP_Y
    \subseteq
    \kappa^{\frac{1}{1-\kappa}}
    \EEE_Y^{-\frac{1}{1-\kappa}}
    =
    \kappa^{\frac{1}{1-\kappa}}
    \supproj\EEE_T^{-\frac{1}{1-\kappa}}
    \subseteq
    \kappa^{\frac{1}{1-\kappa}}
    \supproj\PPP_T^{\frac{1}{1-\kappa}}.
    \qedhere
  \]
\end{proof}

\section{Parametric Bayesian models}
\label{sec:para-Bayesian}

Now we generalize the notion of a Bayesian model
to that of a parametric Bayesian (or \emph{para-Bayesian}) model.
This is a pair consisting of a statistical model $(P_{\theta}\mid\theta\in\Theta)$
on a sample space $\Omega$
and a statistical model $(Q_{\pi}\mid\pi\in\Pi)$
on the sample space $\Theta$
(so that the sample space of the second statistical model
is the parameter space of the first statistical model).
Intuitively, a para-Bayesian model is the counterpart of a Bayesian model
in the situation of uncertainty about the prior:
now the prior is a parametric family of probability measures
rather than one probability measure.

The following definitions are straightforward generalizations
of the definitions for the Bayesian case.
The joint statistical model $T=(T_{\pi}\mid\pi\in\Pi)$
on the measurable space $\Omega\times\Theta$
is defined by
\begin{equation}\label{eq:joint}
  T_{\pi}(A\times B)
  :=
  \int_B P_{\theta}(A) Q_{\pi}(\d\theta),
\end{equation}
for all measurable $A\subseteq\Omega$ and $B\subseteq\Theta$.
For each $\pi\in\Pi$,
$Y_{\pi}$ is the marginal distribution of $T_{\pi}$ on $\Omega$:
for any measurable $A\subseteq\Omega$, $Y_{\pi}(A):=T_{\pi}(A\times\Theta)$.
The \emph{product} $\bar\EEE_P\EEE_Q$ of $\bar\EEE_P$ and $\EEE_Q$ is still defined
as the class of all measurable functions $f:\Omega\times\Theta\to[0,\infty]$ such that,
for some $g\in\bar\EEE_P$ and $h\in\EEE_Q$,
we have the equality in \eqref{eq:product} $T_{\pi}$-a.s., for all $\pi\in\Pi$.

\begin{remark}\label{rem:sufficient}
  Another representation of para-Bayesian models
  is as a sufficient statistic,
  as elaborated in \cite{Lauritzen:1988}:
  \begin{itemize}
  \item
    For the para-Bayesian model $(P,Q)$,
    the statistic $(\theta,\omega)\in(\Theta\times\Omega)\mapsto\theta$
    is a sufficient statistic in the statistical model
    $(T_{\pi})$
    on the product space $\Theta\times\Omega$.
  \item
    If $\theta$ is a sufficient statistic
    for a statistical model $(T_{\pi})$ on a sample space $\Omega$,
    then $(P,Q)$ is a para-Bayesian model,
    where $Q$ is the distribution of $\theta$,
    and $P_{\theta}$ are (fixed versions of) the conditional distributions given $\theta$.
  \end{itemize}
\end{remark}

\begin{remark}
  Yet another way to represent a para-Bayesian model $(P,Q)$
  is a Markov family with time horizon $3$:
  \begin{itemize}
  \item
    the initial state space is $\Pi$,
    the middle one is $\Theta$,
    and the final one is $\Omega$;
  \item
    there is no initial probability measure on $\Pi$,
    the statistical model $(Q_{\pi})$ is the first Markov kernel,
    and the statistical model $(P_{\theta})$ is the second Markov kernel.
  \end{itemize}
\end{remark}

\begin{theorem}\label{thm:VV}
  If $(P,Q)$ is a para-Bayesian model with the joint statistical model $T$
  (as defined by \eqref{eq:joint}),
  we have \eqref{eq:VV-1}.
\end{theorem}

\begin{proof}
  The inclusion ``$\supseteq$'' in~\eqref{eq:VV-1}
  follows from the definition of $T$:
  if $g\in\bar\EEE_P$ and $h\in\EEE_Q$,
  we have, for all $\pi\in\Pi$,
  \begin{align*}
    \int_{\Omega\times\Theta}
    g(\omega;\theta)
    h(\theta)
    T_{\pi}(\d\omega,\d\theta)
    &=
    \int_{\Theta}
    \int_{\Omega}
    g(\omega;\theta)
    P_\theta(\d\omega)
    h(\theta)
    Q_{\pi}(\d\theta)\\
    &\le
    \int_{\Theta}
    h(\theta)
    Q_{\pi}(\d\theta)
    \le
    1.
  \end{align*}

  To check the inclusion ``$\subseteq$'' in~\eqref{eq:VV-1},
  let $f\in\EEE_T$.
  Define $h:\Theta\to[0,\infty]$ and $g:\Omega\times\Theta\to[0,\infty]$ by
  \begin{align*}
    h(\theta)
    &:=
    \int f(\omega,\theta) P_\theta(\d\omega)\\
    g(\omega;\theta)
    &:=
    f(\omega,\theta) / h(\theta)
  \end{align*}
  (setting, e.g., $0/0:=0$ in the last fraction).
  Since by definition,
  $f(\omega,\theta) = g(\omega;\theta)h(\theta)$ $T_{\pi}$-a.s.,
  it suffices to check that $h\in\EEE_Q$ and $g\in\bar\EEE_P$.
  The inclusion $h\in\EEE_Q$ follows
  from the fact that, for any $\pi\in\Pi$,
  \[
    \int_{\Theta} h(\theta) Q_{\pi}(\d\theta)
    =
    \int_{\Theta}
    \int_{\Omega} f(\omega,\theta) P_\theta(\d\omega)
    Q_{\pi}(\d\theta)
    =
    \int_{\Omega\times\Theta}
    f(\omega,\theta)
    T_{\pi}(\d\omega,\d\theta)
    \le
    1.
  \]
  And the inclusion $g\in\bar\EEE_P$ follows
  from the fact that, for any $\theta\in\Theta$,
  \[
    \int g(\omega;\theta) P_{\theta}(\d\omega)
    =
    \int \frac{f(\omega,\theta)}{h(\theta)} P_{\theta}(\d\omega)
    =
    \frac{\int f(\omega,\theta) P_{\theta}(\d\omega)}{h(\theta)}
    =
    \frac{h(\theta)}{h(\theta)}
    \le
    1
  \]
  (we have ${}\le1$ rather than ${}=1$ because of the possibility $h(\theta)=0$).
\end{proof}

\subsection*{IID vs exchangeability}

De Finetti's theorem (see, e.g., \cite[Theorem~1.49]{Schervish:1995})
establishes a close connection between IID and exchangeability
for infinite sequences in $\mathbf{Z}^{\infty}$,
where $\mathbf{Z}$ is a Borel measurable space:
namely, the exchangeable probability measures are the convex mixtures of the IID probability measures
(in particular, their upper envelopes, and therefore, e- and p-functions, coincide).
This subsection discusses a somewhat less close connection
in the case of sequences of a fixed finite length.

Fix $N\in\{1,2,\dots\}$ (time horizon),
and let $\Omega:=\mathbf{Z}^N$ be the set of all sequences of elements of $\mathbf{Z}$
(a measurable space, not necessarily Borel) of length $N$.
An \emph{IID probability measure} on $\Omega$ is a measure of the type $Q^N$,
where $Q$ is a probability measure on $\mathbf{Z}$.
The \emph{configuration} $\conf(\omega)$ of a sequence $\omega\in\Omega$
is the multiset of all elements of $\omega$,
and a \emph{configuration measure} is the pushforward of an IID probability measure on $\Omega$
under the mapping $\conf$.
Therefore, a configuration measure is a measure on the set of all multisets in $\mathbf{Z}$ of size $N$
(with the natural quotient $\sigma$-algebra).

Let $\Eiid$ be the class of all e-functions w.r.\ to the family
of all IID probability measures on $\Omega$
and $\Econf$ be the class of all e-functions w.r.\ to the family
of all configuration probability measures.
Let $\Eexch$ be the class of all e-functions
w.r.\ to the family of all exchangeable probability measures on $\Omega$;
remember that a probability measure $P$ on $\Omega$ is \emph{exchangeable}
if, for any permutation $\pi:\{1,\dots,N\}\to\{1,\dots,N\}$
and any measurable set $E\subseteq\mathbf{Z}^N$,
\[
  P
  \left\{
    (z_1,\dots,z_N)
    \mid
    (z_{\pi(1)},\dots,z_{\pi(N)})
    \in
    E
  \right\}
  =
  P(E).
\]
The \emph{product} $\Eexch\Econf$ of $\Eexch$ and $\Econf$
is the set of all measurable functions $f:\Omega\to[0,\infty]$ such that,
for some $g\in\Eexch$ and $h\in\Econf$,
\[
  f(\omega)
  =
  g(\omega)h(\conf(\omega))
\]
holds for almost all $\omega\in\Omega$
(under any IID probability measure).

\begin{corollary}\label{cor:VV}
  It is true that
  \begin{equation*}
    \Eiid
    =
    \Eexch \Econf.
  \end{equation*}
\end{corollary}

\begin{proof}
  It suffices to apply Theorem~\ref{thm:VV} in the situation
  where $\Theta$ is the set of all configurations,
  $P_{\theta}$ is the probability measure on $\mathbf{Z}^N$
  concentrated on the set of all sequences with the configuration $\theta$
  and uniform on that set
  (we can order $\theta$ arbitrarily,
  and then $P_{\theta}$ assigns weight $1/N!$ to each permutation
  of that ordering),
  $\Pi$ is the set of all IID probability measures on $\Omega$,
  and $Q_{\pi}$ is the pushforward of $\pi\in\Pi$ w.r.\ to the mapping $\conf$.
\end{proof}

\section{Bernoulli sequences: IID vs exchangeability}
\label{sec:Kolmogorov}

In this section we apply the definitions and results of the previous sections
to the problem of defining Bernoulli sequences.
Kolmogorov's main publications on this topic are \cite{Kolmogorov:1968} and \cite{Kolmogorov:1983}.
The results of this section will be algorithm-free versions of the results in \cite{Vovk:1986-arXiv}
(also described in V'yugin's review \cite{Vyugin:1999}, Sections 11--13).

The definitions of the previous subsection simplify as follows.
Now $\Omega:=\{0,1\}^N$ is the set of all binary sequences of length $N$.
Let $\EBern$ be the class of all e-functions w.r.\ to the family
of all Bernoulli IID probability measures on $\Omega$
(this is a special case of $\Eiid$)
and $\Ebin$ be the class of all e-functions w.r.\ to the family
of all binomial probability measures on $\{0,\dots,N\}$
(this is a special case of $\Econf$);
remember that the Bernoulli measure $B_p$ with parameter $p\in[0,1]$
is defined by $B_p(\{\omega\}):=p^k(1-p)^{N-k}$,
where $k:=+\omega$ is the number of 1s in $\omega$,
and the binomial measure $\bin_p$ with parameter $p\in[0,1]$
is defined by $\bin_p(\{k\}):=\binom{N}{k}p^k(1-p)^{N-k}$.
(The notation $+\omega$ for the number $k$ of 1s in $\omega$ is motivated
by $k$ being the sum of the elements of $\omega$.)

We continue to use the notation $\Eexch$ for the class of all e-functions
w.r.\ to the family of all exchangeable probability measures on $\Omega$;
a probability measure $P$ on $\Omega$ is exchangeable
if and only if $P(\{\omega\})$ depends on $\omega$ only via $+\omega$.
It is clear that a function $f:\Omega\to[0,\infty]$ is in $\Eexch$
if and only if, for each $k\in\{0,\dots,N\}$,
\[
  \binom{N}{k}^{-1}
  \sum_{\omega\in\Omega:+\omega=k}
  f(\omega)
  \le
  1.
\]
The product $\Eexch\Ebin$ of $\Eexch$ and $\Ebin$
is the set of all functions $\omega\in\Omega\mapsto g(\omega)h(+\omega)$
for $g\in\Eexch$ and $h\in\Ebin$.
The following is a special case of Corollary~\ref{cor:VV}.

\begin{corollary}\label{cor:V-1}
  It is true that
  \begin{equation*}
    \EBern
    =
    \Eexch \Ebin.
  \end{equation*}
\end{corollary}

The intuition behind Corollary~\ref{cor:V-1}
is that a sequence $\omega\in\Omega$ is Bernoulli
if and only if it is exchangeable and the number of 1s in it is binomial.
The analogue of Corollary~\ref{cor:V-1} in the algorithmic theory of randomness
is Theorem~1 in \cite{Vovk:1986-arXiv},
which says, using the terminology of that paper,
that the Bernoulliness deficiency of $\omega$
equals the binomiality deficiency of $+\omega$
plus the conditional randomness deficiency of $\omega$
in the set of all sequences in $\{0,1\}^N$ with $+\omega$ 1s
given the binomiality deficiency of $+\omega$.
Corollary~\ref{cor:V-1} is simpler since it does not involve any analogue of the condition
``given the binomiality deficiency of $+\omega$''.
Theorem~1 of \cite{Vovk:1986-arXiv} was generalized to the non-binary case
in \cite{Vovk/etal:1999}
(Theorem~3 of \cite{Vovk/etal:1999}, given without a proof, is an algorithmic analogue of Corollary~\ref{cor:VV}).

\begin{remark}
  Kolmogorov's definition of Bernoulli sequences is via exchangeability.
  We can regard this definition as an approximation to definitions taking into account the binomiality of the number of 1s.
  In the paper \cite{Kolmogorov:1968} Kolmogorov uses the word ``approximately''
  when introducing his notion of Bernoulliness
  (p.~663, lines 5--6 after the 4th displayed equation).
  However, it would be wrong to assume that here he acknowledges
  disregarding the requirement that the number of 1s should be binomial;
  this is not what he meant when he used the word ``approximately''~\cite{Kolmogorov:1983pers}.
\end{remark}

The reason for Kolmogorov's definition of Bernoulliness
being different from the definitions based on e-values and p-values
is that $+\omega$ carries too much information about $\omega$;
intuitively \cite{Vovk:1997}, $+\omega$ contains not only useful information
about the probability $p$ of 1 but also noise.
To reduce the amount of noise,
we will use an imperfect estimator of $p$.
Set
\begin{equation}\label{eq:partition}
  p(a)
  :=
  \sin^2
  \left(
    a N^{-1/2}
  \right),
  \quad
  a=1,\dots,N^*-1,
  \quad
  N^*
  :=
  \left\lfloor
    \frac{\pi}{2}N^{1/2}
  \right\rfloor,
\end{equation}
where $\lfloor\cdot\rfloor$ stands for integer part.
Let $E:\{0,\dots,N\}\to[0,1]$ be the estimator of $p$ defined by
$
  E(k)
  :=
  p(a)
$,
where $p(a)$ is the element of the set \eqref{eq:partition} that is nearest to $k/N$
among those satisfying $p(a)\le k/N$;
if such elements do not exist, set $E(k):=p(1)$.

Denote by $\mathfrak{A}$ the partition of the set $\{0,\dots,N\}$ into the subsets
$E^{-1}(E(k))$,
where $k\in\{0,\dots,N\}$.
For any $k\in\{0,\dots,N\}$,
$\mathfrak{A}(k):=E^{-1}(E(k))$ denotes the element of the partition $\mathfrak{A}$ containing $k$.
Let $\Esin$ be the class of all e-functions
w.r.\ to the statistical model $\{U_k\mid k\in\{0,\dots,N\}\}$,
$U_k$ being the uniform probability measure on $\mathfrak{A}(k)$.
(This is a Kolmogorov-type statistical model,
consisting of uniform probability measures on finite sets;
see, e.g., \cite[Section~4]{Vovk/Shafer:2003}.)

\begin{theorem}\label{thm:V-2}
  For some universal constant $c>0$,
  \begin{equation*}
    c^{-1}
    \Esin
    \subseteq
    \Ebin
    \subseteq
    c \Esin.
  \end{equation*}
\end{theorem}

The analogue of Theorem~\ref{thm:V-2} in the algorithmic theory of randomness
is Theorem~2 in \cite{Vovk:1986-arXiv},
and the proof of Theorem~\ref{thm:V-2} can be extracted
from that of Theorem~2 in \cite{Vovk:1986-arXiv} (details omitted).

\begin{remark}
  Paper \cite{Vovk:1986-arXiv} uses a net slightly different from \eqref{eq:partition};
  \eqref{eq:partition} was introduced in \cite{Vovk:1997} and also used in \cite{Vyugin:1999}.
\end{remark}

In conclusion of this section,
let us extract corollaries in terms of p-values
from Corollary~\ref{cor:V-1} and Theorem~\ref{thm:V-2};
we will use the obvious notation 
$\PBern$, $\Pexch$, and $\Pbin$.

\begin{corollary}\label{cor:V-1-p}
  For each $\kappa\in(0,1)$,
  \begin{equation}\label{eq:V-1-p}
    \kappa^{-1} \PBern^{1-\kappa}
    \subseteq
    \Pexch \Pbin
    \subseteq
    \kappa^{\frac{2}{1-\kappa}} \PBern^{\frac{1}{1-\kappa}}.
  \end{equation}
\end{corollary}

\begin{proof}
  Similarly to Corollary~\ref{cor:VV-1},
  the left inclusion of \eqref{eq:V-1-p} follows from
  \[
    \kappa^{-1} \PBern^{1-\kappa}
    \subseteq
    \EBern^{-1}
    =
    \Eexch^{-1} \Ebin^{-1}
    \subseteq
    \Pexch \Pbin,
  \]
  and the right inclusion of \eqref{eq:V-1-p} follows from
  \begin{equation*}
    \Pexch \Pbin
    \subseteq
    \kappa^{\frac{2}{1-\kappa}}
    (\Eexch \Ebin)^{-\frac{1}{1-\kappa}}
    =
    \kappa^{\frac{2}{1-\kappa}}
    \EBern^{-\frac{1}{1-\kappa}}
    \subseteq
    \kappa^{\frac{2}{1-\kappa}}
    \PBern^{\frac{1}{1-\kappa}}.
    \qedhere
  \end{equation*}
\end{proof}

\begin{corollary}\label{cor:V-2-p}
  There is a universal constant $c>0$ such that,
  for each $\kappa\in(0,0.9)$,
  \begin{equation}\label{eq:V-2-p}
    c
    \kappa^{-1} \Psin^{1-\kappa}
    \subseteq
    \Pbin
    \subseteq
    c^{-1}
    \kappa^{\frac{1}{1-\kappa}} \Psin^{\frac{1}{1-\kappa}}.
  \end{equation}
\end{corollary}

\begin{proof}
  As in the previous proof,
  the left inclusion of \eqref{eq:V-2-p} follows from
  \[
    \kappa^{-1} \Psin^{1-\kappa}
    \subseteq
    \Esin^{-1}
    \subseteq
    c^{-1}
    \Ebin^{-1}
    \subseteq
    c^{-1}
    \Pbin,
  \]
  and the right inclusion from
  \[
    \Pbin
    \subseteq
    \kappa^{\frac{1}{1-\kappa}}
    \Ebin^{-\frac{1}{1-\kappa}}
    \subseteq
    c^{-1}
    \kappa^{\frac{1}{1-\kappa}}
    \Esin^{-\frac{1}{1-\kappa}}
    \subseteq
    c^{-1}
    \kappa^{\frac{1}{1-\kappa}}
    \Psin^{\frac{1}{1-\kappa}},
  \]
  where $c$ stands for a positive universal constant.
\end{proof}

\section{Conclusion}

In this section we discuss some directions of further research.
A major advantage of the non-algorithmic approach to randomness proposed in this paper
is the absence of unspecified constants;
in principle, all constants can be computed.
The most obvious open problem is to find the best constant $c$ in Theorem~\ref{thm:V-2}.

In Section~\ref{sec:Kolmogorov} we discussed a possible implementation
of Kolmogorov's idea of defining Bernoulli sequences.
However, Kolmogorov's idea was part of a wider programme;
e.g., in \cite[Section 5]{Kolmogorov:1983} he sketches a way of applying a similar approach to Markov sequences.
For other possible applications,
see \cite[Section~4]{Vovk/Shafer:2003}
(most of these applications were mentioned by Kolmogorov in his papers and talks).
Analogues of Corollary~\ref{cor:V-1} in Section~\ref{sec:Kolmogorov}
can be established for these other applications
(cf.\ \cite{Lauritzen:1988} and Remark~\ref{rem:sufficient}),
but it is not obvious whether Theorem~\ref{thm:V-2} can be extended
in a similar way.


\begin{thebibliography}{10}

\bibitem{Bourbaki:integration}
Nicolas Bourbaki.
\newblock {\em Elements of Mathematics. Integration}.
\newblock Springer, Berlin, 2004.
\newblock In two volumes. The French originals published in 1952--1969.

\bibitem{Gacs:2005}
Peter G\'acs.
\newblock Uniform test of algorithmic randomness over a general space.
\newblock {\em Theoretical Computer Science}, 341:91--137, 2005.
\newblock {A} later version of this paper (2013) is available on the author's
  web site (accessed in October 2019).

\bibitem{Gammerman/Vovk:2003}
Alex Gammerman and Vladimir Vovk.
\newblock Data labelling apparatus and method thereof, 2003.
\newblock {US} Patent Application 0236578 A1. Available on the Internet
  (accessed in October 2019).

\bibitem{Grunwald/etal:arXiv1906}
Peter Gr\"unwald, Rianne de~Heide, and Wouter~M. Koolen.
\newblock Safe testing.
\newblock Technical Report
  \href{https://arxiv.org/abs/1906.07801}{arXiv:1906.07801 [math.ST]},
  \href{https://arxiv.org/}{arXiv.org} e-Print archive, June 2019.

\bibitem{Gurevich/Vovk:2019COPA}
Yuri Gurevich and Vladimir Vovk.
\newblock Test statistics and p-values.
\newblock {\em Proceedings of Machine Learning Research}, 105:89--104, 2019.
\newblock {COPA} 2019.

\bibitem{Kolmogorov:1968-talk}
Andrei~N. Kolmogorov.
\newblock \begin{cyr}Несколько теорем об
  алгоритмической энтропии и
  алгоритмическом количестве
  информации\end{cyr}.
\newblock {\em \begin{cyr}Успехи математических
  наук\end{cyr}}, 23(2):201, 1968.
\newblock Abstract of a talk before the Moscow Mathematical Society. Meeting of
  31 October 1967.

\bibitem{Kolmogorov:1968}
Andrei~N. Kolmogorov.
\newblock Logical basis for information theory and probability theory.
\newblock {\em IEEE Transactions on Information Theory}, IT-14:662--664, 1968.
\newblock {R}ussian original: \begin{cyr}К логическим основам
  теории информации и теории
  вероятностей\end{cyr}, published in \begin{cyr}Проблемы
  передачи информации\end{cyr}.

\bibitem{Kolmogorov:1983}
Andrei~N. Kolmogorov.
\newblock Combinatorial foundations of information theory and the calculus of
  probabilities.
\newblock {\em Russian Mathematical Surveys}, 38:29--40, 1983.
\newblock {R}ussian original: \begin{cyr}Комбинаторные
  основания теории информации и исчисления
  вероятностей\end{cyr}.

\bibitem{Kolmogorov:1983LNM-local}
Andrei~N. Kolmogorov.
\newblock On logical foundations of probability theory.
\newblock In Yu.~V. Prokhorov and K.~It\^o, editors, {\em Probability Theory
  and Mathematical Statistics}, volume 1021 of {\em Lecture Notes in
  Mathematics}, pages 1--5. Springer, 1983.
\newblock Talk at the Fourth USSR--Japan Symposium on Probability Theory and
  Mathematical Statistics (Tbilisi, August 1982) recorded by Alexander A.
  Novikov, Alexander K. Zvonkin, and Alexander Shen. Our quote follows
  \emph{Selected Works of A.~N. Kolmogorov}, volume II, Probability Theory and
  Mathematical Statistics, edited by A.~N.~Shiryayev, Kluwer, Dordrecht,
  p.~518.

\bibitem{Kolmogorov:1983pers}
Andrei~N. Kolmogorov.
\newblock Personal communication, ca 1983.

\bibitem{Lauritzen:1988}
Steffen~L. Lauritzen.
\newblock {\em Extremal Families and Systems of Sufficient Statistics},
  volume~49 of {\em Lecture Notes in Statistics}.
\newblock Springer, New York, 1988.

\bibitem{Levin:1976-local}
Leonid~A. Levin.
\newblock Uniform tests of randomness.
\newblock {\em Soviet Mathematics Doklady}, 17:337--340, 1976.
\newblock {R}ussian original: \begin{cyr}Равномерные тесты
  случайности\end{cyr}.

\bibitem{Li/Vitanyi:2008}
Ming Li and Paul Vit\'anyi.
\newblock {\em An Introduction to {K}olmogorov Complexity and Its
  Applications}.
\newblock Springer, New York, third edition, 2008.

\bibitem{Martin-Lof:1966}
Per Martin-L\"of.
\newblock The definition of random sequences.
\newblock {\em Information and Control}, 9:602--619, 1966.

\bibitem{Schervish:1995}
Mark~J. Schervish.
\newblock {\em Theory of Statistics}.
\newblock Springer, New York, 1995.

\bibitem{Shafer:arXiv1903}
Glenn Shafer.
\newblock The language of betting as a strategy for statistical and scientific
  communication.
\newblock Technical Report
  \href{https://arxiv.org/abs/1903.06991}{arXiv:1903.06991 [math.ST]},
  \href{https://arxiv.org/}{arXiv.org} e-Print archive, March 2019.

\bibitem{Shafer/Vovk:2019}
Glenn Shafer and Vladimir Vovk.
\newblock {\em Game-Theoretic Foundations for Probability and Finance}.
\newblock Wiley, Hoboken, NJ, 2019.

\bibitem{Suslin:1917}
Mikhail~Ya. Souslin.
\newblock Sur une d\'efinition des ensembles mesurables {B} sans nombres
  transfinis.
\newblock {\em Comptes rendus hebdomadaires des s\'eances de l'Acad\'emie des
  sciences}, 164:88--91, 1917.

\bibitem{Vovk:1986-arXiv}
Vladimir Vovk.
\newblock On the concept of the {B}ernoulli property.
\newblock {\em Russian Mathematical Surveys}, 41:247--248, 1986.
\newblock {R}ussian original: \begin{cyr}О понятии
  бернуллиевости\end{cyr}. Another English translation with
  proofs: \cite{Vovk:arXiv1612}.

\bibitem{Vovk:1997}
Vladimir Vovk.
\newblock Learning about the parameter of the {B}ernoulli model.
\newblock {\em Journal of Computer and System Sciences}, 55:96--104, 1997.

\bibitem{Vovk:arXiv1612}
Vladimir Vovk.
\newblock On the concept of {B}ernoulliness.
\newblock Technical Report
  \href{https://arXiv.org/abs/1612.08859}{arXiv:\allowbreak 1612.\allowbreak
  08859} [math.ST], \href{https://arXiv.org}{arXiv.org} e-Print archive,
  December 2016.

\bibitem{Vovk/etal:1999}
Vladimir Vovk, Alex Gammerman, and Craig Saunders.
\newblock Machine-learning applications of algorithmic randomness.
\newblock In {\em Proceedings of the Sixteenth International Conference on
  Machine Learning}, pages 444--453, San Francisco, CA, 1999. Morgan Kaufmann.

\bibitem{Vovk/Shafer:2003}
Vladimir Vovk and Glenn Shafer.
\newblock Kolmogorov's contributions to the foundations of probability.
\newblock {\em Problems of Information Transmission}, 39:21--31, 2003.

\bibitem{Vovk/Vyugin:1993}
Vladimir Vovk and Vladimir~V. V'yugin.
\newblock On the empirical validity of the {B}ayesian method.
\newblock {\em Journal of the Royal Statistical Society B}, 55:253--266, 1993.

\bibitem{Vyugin:1999}
Vladimir~V. V'yugin.
\newblock Algorithmic complexity and stochastic properties of finite binary
  sequences.
\newblock {\em Computer Journal}, 42:294--317, 1999.
\end{thebibliography}
\end{document}